\documentclass[11pt,letterpaper]{amsart}

\usepackage{amsfonts, amsmath, amssymb, amsgen, amsthm, amscd}
\usepackage{latexsym,mathrsfs}

\usepackage{enumerate}

\usepackage{color}
\usepackage[all]{xy}

\usepackage[utf8]{inputenc} 

\usepackage[top=3.5cm, left=2.75cm, right=2.75cm, bottom=2.75cm]{geometry}

\def\Id{\mathop{\rm Id}\nolimits}

\def\Ad{\mathop{\rm Ad}\nolimits}

\def\cotor{\mathop{\rm Cotor}\nolimits}

\def\Hom{\mathop{\rm Hom}\nolimits}

\def\Rb{{\mathbb R}}

\def\Zb{{\mathbb Z}}

\def\Ac{{\mathcal A}}

\def\Hc{{\mathcal H}}

\def\Kc{{\mathcal K}}

\def\Cc{{\mathcal C}}
\def\Zc{{\mathcal Z}}

\def\Oc{{\mathcal O}}

\def\Dc{{\mathcal D}}

\def\a{\alpha}

\def\d{\delta}
\def\D{\Delta}

\def\s{\sigma}

\def\ve{\varepsilon}
\def\vp{\varphi}

\def\0b{\bf 0}

\def\nb{\nabla}
\def\ot{\otimes}

\def\ra{\rightarrow}

\def\rt{\triangleright}

\def\lt{\triangleleft}

\def\p{\partial}

\def\0D{\Delta^{(0)}}
\def\1D{\Delta^{(1)}}

\newcommand{\wbar}[1]{\overline{#1}}

\newcommand{\Fb}{\mathfrak{b}}

\newcommand{\Fg}{\mathfrak{g}}
\newcommand{\Fh}{\mathfrak{h}}

\def\build#1_#2^#3{\mathrel{
\mathop{\kern 0pt#1}\limits_{#2}^{#3}}}
\newcommand{\ps}[1]{~\hspace{-4pt}_{^{(#1)}}}

\newcommand{\ns}[1]{~\hspace{-4pt}_{_{{<#1>}}}}

\newcommand{\snb}[1]{~\hspace{-4pt}^{^{{[#1]}}}}

\def\odots{\ot\cdots\ot}
\def\wdots{\wedge\dots\wedge}

\def\one{{\bf 1}}

\newcommand{\ie}{{\it i.e.\/}\ }

\def\a{\alpha}

\def\d{\delta}

\def\s{\sigma}

\def\ve{\varepsilon}

\def\vp{\varphi}

\def\D{\Delta}

\def\dt{\left.\frac{d}{dt}\right|_{_{t=0}}}

\def\nb{\nabla}

\def\ot{\otimes}
\def\part{\partial}

\def\ra{\rightarrow}

\def\lra{\leftrightarrow}
\def\text{\hbox}

\def\nb{\nabla}
\def\ot{\otimes}

\def\ra{\rightarrow}

\def\Ad{\mathop{\rm Ad}\nolimits}

\def\Ext{\mathop{\rm Ext}\nolimits}
\def\Hom{\mathop{\rm Hom}\nolimits}

\def\Id{\mathop{\rm Id}\nolimits}
\def\exp{\mathop{\rm exp}\nolimits}

\def\Tor{\mathop{\rm Tor}\nolimits}

\def\lra{\longrightarrow}

\def\build#1_#2^#3{\mathrel{
\mathop{\kern 0pt#1}\limits_{#2}^{#3}}}

\newcommand{\CH}{\text{\bf CH}}

\newcommand{\CB}{\text{\bf CB}}
\newcommand{\C}[1]{\mathcal{#1}}
\newcommand{\B}[1]{\mathbb{#1}}

\renewcommand{\leq}{\leqslant}
\renewcommand{\geq}{\geqslant}

\numberwithin{equation}{section}
\newtheorem{theorem}{Theorem}[section]
\newtheorem{proposition}[theorem]{Proposition}
\newtheorem{lemma}[theorem]{Lemma}
\newtheorem{corollary}[theorem]{Corollary}

\theoremstyle{definition}
\newtheorem{remark}[theorem]{Remark}

\newtheorem{definition}[theorem]{Definition}

\title{A characteristic map for compact quantum groups}

\author{Atabey Kaygun} 
\address{Istanbul Technical University, Department of Mathematics, Istanbul, Turkey}
\email{atabey.kaygun@gmail.com}

\author{Serkan Sütlü}
\address{Işık University, Department of Mathematics, Istanbul, Turkey}
\email{serkan.sutlu@isikun.edu.tr}

\begin{document}
\maketitle


\begin{abstract}
  We show that if $G$ is a compact Lie group and $\Fg$ is its Lie
  algebra, then there is a map from the Hopf-cyclic cohomology of the
  quantum enveloping algebra $U_q(\Fg)$ to the twisted cyclic
  cohomology of quantum group algebra $\Oc(G_q)$.  We also show that
  the Schmüdgen-Wagner index cocycle associated with the volume form
  of the differential calculus on the standard Podleś sphere
  $\Oc(S^2_q)$ is in the image of this map.
\end{abstract}

\section*{Introduction}

Given a compact Lie group $G$ and its Lie algebra $\mathfrak{g}$,
there is a characteristic map of the form
$HP^*(U(\mathfrak{g}),k_\d) \lra HP^*(\C{O}(G))$ coming from the
Connes-Moscovici theory~\cite{ConnMosc98}.  Here, the domain of the
map depends on the Lie algebra homology of $\mathfrak{g}$, and the
range is the ordinary periodic algebra cyclic cohomology of the
algebra of regular functions on $G$, which depends on the de~Rham
homology of $G$.  We refer the reader to
Subsection~\ref{ClassicalVanEst} for details.  In this paper we
develop a $q$-analogue of this map.  To be precise, in
Theorem~\ref{MainResult1} we show that for a compact quantum group
algebra $\mathcal{O}(G_q)$ and its quantum enveloping algebra
$U_q(\mathfrak{g})$ there is a morphism in cohomology of the form
\begin{equation}\label{aux-q-vanest}
HC^\ast(U_q(\mathfrak{g}),{}^\sigma k) \lra
   HC^\ast_{\sigma^{-1}}(\mathcal{O}(G_q))
\end{equation}
whose domain is the Hopf-cyclic cohomology of $U_q(\mathfrak{g})$ with
coefficients in the modular pair in involution (MPI) determined by
\cite[Prop. 6.1.6]{KlimSchm-book}, and whose range is the twisted
cyclic cohomology of $\mathcal{O}(G_q)$ viewed as an algebra.  Recall
that the Connes-Moscovici characteristic map can be viewed as a cup
product~\cite{KhalRang05-II,Rang08, Kayg08, Kayg11}.  Our key
observation is that when we write the analogous cup product using the
Haar functional of a compact quantum group, the modularity property of
the Haar functional~\cite[Prop. 11.34]{KlimSchm-book} gives us the
twisted algebra cyclic cohomology in the range in~\eqref{aux-q-vanest}
in contrast to the Connes-Moscovici case where the range is the
ordinary algebra cyclic cohomology.  We further observe that one can
\emph{untwist} the cohomology with an appropriate additional cup
product, but this procedure brings in a degree shift.  The shift
coming from the untwisting cup product explains the dimension drop
phenomenon observed in \cite{FengTsyg91}, and the degree shift
phenomenon observed in~\cite{GoodKrah14}.  We refer the reader to
Subsection~\ref{Untwisting} and Corollary~\ref{MainResult2}.

We show the non-triviality of the characteristic homomorphism
\eqref{aux-q-vanest} in Section~\ref{subsect-van-est}.  We first
recall that in \cite{MasuNakaWata90} Masuda, Nakagami and Watanabe
calculated the classical Hochschild and cyclic cohomology of
$\mathcal{O}(SL_q(2))$ using an explicit resolution.  Then, in
Proposition~\ref{MNW1} and Corollary~\ref{MNW2} we recover one
specific generator of the cyclic cohomology of $\Oc(SL_q(2))$ in the
image of the characteristic homomorphism \eqref{aux-q-vanest}.

In analogy with the fact that Connes-Moscovici characteristic map
allows the index computation of codimension-$n$ foliations to take
place in the Hopf-cyclic cohomology of the Hopf algebra $\Hc_n$ of
codimension-$n$ foliations, we introduce \eqref{aux-q-vanest}
to pull the index computation on the twisted cyclic cohomology of
$\Oc(G_q)$ to the Hopf-cyclic cohomology
$HC^\ast(U_q(\mathfrak{g}),{}^\sigma k)$ of the Hopf algebra of
$U_q(\mathfrak{g})$, which was computed in \cite{KaygSutl14}. This
fact, along with the quantum homogeneous space version of our
characteristic map we develop in
Section~\ref{quantumHomogeneousSpaces}, turns the Hopf-cyclic
cohomology of quantum groups into a useful tool detecting the index
cocycles of such spaces.  We use the equivariant characteristic map of
\cite{RangSutl-IV} to show that in the case of the standard Podleś
sphere, the characteristic homomorphism \eqref{aux-q-vanest} descends
to
\begin{equation}\label{aux-q-vanest-SU}
HC^\ast_{k[\sigma,\sigma^{-1}]}(U_q(su_2), {}^{\sigma^{-1}} k, {}^{\sigma} k) \lra HC^\ast_{\sigma^{-1}}(\mathcal{O}(S_q^2)),
\end{equation}
where the domain is now the equivariant Hopf-cyclic cohomology of the
quantum enveloping algebra $U_q(su_2)$.  Moreover, we show that a
$\s^{-1}$-twisted version of the Schmüdgen-Wagner quantum index
cocycle of \cite{SchmWagn04}, see also \cite{Hadf07}, that computes
the index of the Dirac operator on $\mathcal{O}(SU_q(2))$ is in the
image of \eqref{aux-q-vanest-SU}.  In the particular case of the
(standard) Podleś sphere $\Oc(S_q^2)\subseteq \Oc(SU_q(2))$, we
further realize the Schmüdgen-Wagner index cocycle in the
equivariant Hopf-cyclic cohomology \cite{RangSutl-IV} of $U_q(su_2)$.

\subsection*{Notation and conventions}

We use a base field $k$ of characteristic 0.  WLOG one can assume
$k=\B{R}$.  We are going to use $\CB$ and $\CH$ to denote respectively
the bar and the Hochschild complexes associated with a (co)cyclic
module.  In the same vein, we use $HH$, $HC$ and $HP$ to denote
respectively the Hochschild, the cyclic and the periodic cyclic
(co)homology of a (co)cyclic module.  We are going to use $\cotor_C^*$
to denote the right derived functor of the (left exact) monoidal
product $\Box_C$ in the category of $C$-comodules of a coassociative
counital coalgebra $C$.  A coextension $\pi\colon C\to D$ is an
epimorphism of (counital) coalgebras.

\subsection*{Acknowledgments}

We would like to thank the anonymous referee whose careful reading and
numerous suggestions greatly improved mathematical content and
exposition of the article.  We are grateful to the referee for
alerting us about the existence of the unpublished note~\cite{KRT1}.

\section{Preliminaries}

In this section we recall the basic material that will be needed in
the sequel. 

\subsection{Cobar and Hochschild complexes}
In this subsection we recall the definition of the cobar complex of a coalgebra $\Cc$, as well as the $\cotor$-groups associated to a coalgebra $\Cc$ and a pair 
$(V,W)$ of $\Cc$-comodules of opposite parity.

Let $\Cc$ be a coassociative coalgebra. Following \cite{BrzeWisb-book,Doi81} and \cite{KaygKhal06}, the cobar complex of
$\Cc$ is defined to be the differential graded space
\begin{equation*}
{\CB}^\ast(\Cc) := \bigoplus_{n\geq0}\Cc^{\ot n+2}
\end{equation*}
with the differentials $d\colon {\CB}^n(\Cc) \longrightarrow {\CB}^{n+1}(\Cc)$
\begin{align*}
d(c^0\odots c^{n+1})=\sum_{j=0}^{n}(-1)^j\,c^0\odots \D(c^j)\odots c^{n+1}.
\end{align*}

Let $\Cc^e:= \Cc \ot \Cc^{\rm cop}$ be the enveloping coalgebra of $\Cc$. In case $\Cc$ is counital, the cobar complex ${\CB}^\ast(\Cc)$ yields a $\Cc^e$-injective resolution of the (left) $\Cc^e$-comodule $\Cc$, see for instance \cite{Doi81}.

Following the terminology of \cite{Kayg12}, for a pair $(V,W)$ of two $\Cc$-comodules of opposite parity (say, $V$ is a right $\Cc$-comodule and $W$ is a left $\Cc$-comodule,) we call the complex
\begin{equation*}
\left({\CB}^\ast(V,\,\Cc,\,W),\,d\right), \quad {\CB}^\ast(V,\,\Cc,\,W):=V\Box_\Cc {\CB}^\ast(\Cc)\Box_\Cc W
\end{equation*}
where we define $d\colon {\CB}^n(V,\,\Cc,\,W)\lra {\CB}^{n+1}(V,\,\Cc,\,W)$
\begin{align}\label{aux-cobar}
d(v \ot c^1 \ot \ldots\ot c^n \ot w) 
= &  v\ns{0}\ot v\ns{1}\ot c^1\odots c^{n}\ot w \nonumber\\
  & + \sum_{j=1}^{n}(-1)^j\,c^1\odots \D(c^j)\odots c^n \ot w \\
  & + (-1)^{n+1}\,v\ot c^1\odots c^{n}\ot w\ns{-1}\ot w\ns{0}\nonumber
\end{align}
the two-sided (cohomological) cobar complex of the coalgebra $\Cc$.
In case $\Cc$ is a counital coalgebra, the $\cotor$-groups of a pair
$(V,W)$ of $\Cc$-comodules of opposite parity can be computed from
\begin{equation}\label{aux-cotor-by-cobar}
\cotor_\Cc^\ast(V,W)=H^\ast({\CB}^\ast(V,\,\Cc,\,W),\,d).
\end{equation}

We next recall the Hochschild cohomology of a coalgebra $\Cc$ with coefficients in the $\Cc$-bicomodule (equivalently $\Cc^e$-comodule) $V$, from \cite{Doi81}, as the cohomology of the complex
\begin{equation*}
\CH^\ast(\Cc,V) = \bigoplus_{n\geq 0} \CH^n(\Cc,V),\qquad\CH^n(\Cc,V):=V\ot \Cc^{\ot \,n}
\end{equation*}
with the differential  $b\colon\CH^n(\Cc,V) \lra \CH^{n+1}(\Cc,V)$ defined as
\begin{align}\label{aux-Hochschild-differential}
b(v\ot c^1\odots c^n) 
= & v\ns{0}\ot v\ns{1} \ot c^1 \odots c^n \nonumber\\
  & + \sum_{k= 1}^n(-1)^kc^1\odots \D(c^k)\odots c^n\\
  & + (-1)^{n+1} v\ns{0}\ot c^1\odots c^n\ot v\ns{-1}.\nonumber
\end{align}
We identify $\CB^n(\Cc)$ with $\Cc^e \ot \Cc^{\ot \,n}$ as left
$\Cc^e$-comodules for $n>0$ via
\begin{equation*}
c^0\odots c^{n+1} \mapsto (c^0\ot c^{n+1})\ot c^1\odots c^n.
\end{equation*}
The left $\Cc^e$-comodule structure on $\CB^n(\Cc) = \Cc^{\ot\,n+2}$
is given by
$$\nb(c^0\odots c^{n+1})=(c^0\ps{1}\ot c^{n+1}\ps{2})\ot(c^0\ps{2}\ot
c^1\odots c^n\ot c^{n+1}\ps{1}),$$
and on $\Cc^e\ot\Cc^{\ot\,n}$ by
$$\nb((c\ot c')\ot(c^1\odots c^n))=(c\ps{1}\ot c'\ps{2})\ot(c\ps{2}\ot
c'\ps{1})\ot(c^1\odots c^n).$$ This yields an isomorphism of the form
\begin{equation*}
\left(\CH^\ast(\Cc,V),\,b\right) \cong \left(V\Box_{\Cc^e}\CB^\ast(\Cc),\,d\right)
\end{equation*}
on the chain level. In case $\Cc$ is counital one can interpret the
Hochschild cohomology of $\Cc$ with coefficients in $V$ in terms of
$\cotor$-groups as
\begin{equation*}
H^\ast(\Cc,V) = H^\ast(\CH^\ast(\Cc,V),\,b) = \cotor^\ast_{\Cc^e}(V,\Cc),
\end{equation*}
or more generally,
\begin{equation*}
H^\ast(\Cc,V) = H^\ast(V\Box_{\Cc^e}Y^*)
\end{equation*}
for any coflat resolution $Y^*$ of $\Cc$ via left $\Cc^e$-comodules.

\subsection{Cohomology of coextensions}\label{subsect-cohom-coext}
In this subsection we recall the main computational tool introduced in
\cite{KaygSutl14} which can be summarized as follows: Given a coflat
coalgebra coextension $\Cc\lra \Dc$, the coalgebra Hochschild
cohomology of $\Cc$ can be computed relatively easily by means of the
Hochschild cohomology of $\Dc$, and the relative cohomology of the
coextension.

Let $\pi:\Cc\lra \Dc$ be a coextension. We first introduce the
auxiliary coalgebra $\Zc:= \Cc\oplus \Dc$ with the comultiplication
\begin{equation*}
\Delta(y) = y_{(1)}\otimes y_{(2)} \quad\text{ and }\quad
   \Delta(x) = x_{(1)}\otimes x_{(2)} + \pi(x_{(1)})\otimes x_{(2)} + x_{(1)}\otimes \pi(x_{(2)}),
\end{equation*}
and the counit
\begin{equation*}
\ve(x+y) = \ve(y),
\end{equation*}
for any $x\in \Cc$ and $y\in \Dc$.

Next, let $V$ be a $\Cc$-bicomodule and let $\Cc$ be coflat both as a
left and a right $\Dc$-comodule. Then via the short exact sequence
\begin{equation*}
  0 \to \Dc \xrightarrow{\ i\ } \Zc \xrightarrow{\ p\ } \Cc \to 0
  \qquad \text{ where }
  \qquad i: y\mapsto (0,y)
  \qquad p: (x,y)\mapsto x
\end{equation*}
of coalgebras and \cite[Lemma 4.10]{FariSolo00}, we have
\begin{equation}
HH^n(\Zc,V)\cong HH^n(\Cc,V), \qquad n\geq 0.
\end{equation}

We next consider $\CH^\ast(\Zc,V)$ with the decreasing filtration
\begin{equation*}
F^{n+p}_p =  \left\{ \begin{array}{ll}
          \bigoplus_{n_0+\cdots+n_p=n} V\otimes \Zc^{\otimes n_0}\otimes \Cc\otimes\cdots\otimes \Zc^{n_{p-1}}\otimes \Cc\otimes \Zc^{\otimes n_p}, & p\geq 0 \\
          0, & p<0.
        \end{array}
\right.
\end{equation*}
The associated spectral sequence is
\begin{equation*}
E^{i,j}_0 = F^{i+j}_i / F^{i+j}_{i+1}
   = \bigoplus_{n_0+\cdots+n_i=j} V\otimes \Dc^{\otimes n_0}\otimes \Cc\otimes\cdots\otimes \Dc^{n_{i-1}}\otimes \Cc\otimes \Dc^{\otimes n_i},
\end{equation*}
and by the coflatness assumption, on the vertical direction it computes
\begin{equation*}
HH^j(\Dc,\Cc^{\Box_\Dc\,i}\,\Box_\Dc\,V).
\end{equation*}

Hence we have the following.

\begin{theorem}\label{thm-spec-seq}
Let $\pi:\Cc\lra \Dc$ be a coalgebra coextension and $V=V'\ot V''$ a $\Cc$-bicomodule such that the left $\Cc$-comodule structure is given by $V'$ and the right $\Cc$-comodule structure is given by $V''$. Let also $\Cc$ be coflat both as a left and a right $\Dc$-comodule. Then there is a spectral sequence whose $E_1$-term is of the form
\begin{equation*}
E_1^{i,j}= \cotor^j_\Dc(V'', \Cc^{\Box_\Dc\,i}\,\Box_\Dc\,V'),
\end{equation*}
converging to $HH^{i+j}(\Cc,V)$.
\end{theorem}

A convenient set-up as a test case for our machinery is a principal
coextension \cite{Schn90,BrzeHaja09}. We assume $\Hc$ is a Hopf
algebra with a bijective antipode, and $\Cc$ is a left $\Hc$-module
coalgebra. Since $\Hc^+ = \ker\ve$ is the augmentation ideal of $\Hc$,
if we define a quotient coalgebra by $\Dc:=\Cc/\Hc^{+}\Cc$ then by
\cite[Theorem II]{Schn90},
\begin{enumerate}
\item $\Cc$ is a projective left $\Hc$-module,
\item ${\rm can}:\Hc\ot \Cc\lra \Cc\Box_\Dc \Cc$, $h\ot c\mapsto h\cdot c\ps{1}\ot c\ps{2}$ is injective,
\end{enumerate}
if and only if
\begin{enumerate}
\item $\Cc$ is faithfully flat left (and right) $\Dc$-comodule,
\item ${\rm can}:\Hc\ot \Cc\lra \Cc\Box_\Dc \Cc$ is an isomorphism.
\end{enumerate}

\subsection{Cyclic cohomology of algebras}

We recall the cyclic cohomology of algebras \cite{Conn83,Conn85,Connes-book,Loday-book}. Let $\Ac$ be an algebra and $M$ an $\Ac$-bimodule. Then the Hochschild cohomology $HH(\Ac,M)$ of $\Ac$ with coefficients in $M$ is the homology of the complex
\begin{equation}
C(\Ac,M)=\bigoplus_{n\geq 0} C^n(\Ac,M),
\end{equation}
where $C^n(\Ac,M)$ is the space of all linear maps $\Ac^{\ot\,n}\lra M$ with the differential
\begin{align}
 b\vp(a_1,\ldots,a_{n+1}) 
= & a_1\cdot \vp(a_2,\ldots,a_{n+1})\nonumber\\
  & + \sum_{k=1}^n(-1)^k\vp(a_1,\ldots,a_ka_{k+1}\ldots,a_{n+1}) \\
  & + (-1)^{n+1}\vp(a_1,\ldots,a_n)\cdot a_{n+1}.\nonumber
\end{align}

The space $\Ac^\ast$ of all linear maps of the form $\Ac\to k$ is an
$\Ac$-bimodule via $a\cdot \vp\cdot b (c)= \vp(bca)$ defined for every
$\vp\in \Ac^\ast$, and $a,b,c\in\Ac$. Hence, the complex
$C(\Ac,\Ac^\ast)$ can be defined. If we identify
$\vp\in C^n(\Ac,\Ac^\ast)$ with
\begin{equation}\label{aux-identification}
\phi:\Ac^{\ot\,n+1}\lra k,\qquad \phi(a_0,\a_1,\ldots,a_n):=\vp(a_1,\ldots,a_n)(a_0),
\end{equation}
the coboundary map corresponds to
\begin{align}
b\phi(a_0,\ldots,a_{n+1}) 
= & \phi(a_0a_1,\ldots,a_{n+1})\nonumber\\
  & + \sum_{k=1}^n(-1)^k\vp(a_0,\ldots,a_ka_{k+1}\ldots,a_{n+1}) \\
  & + (-1)^{n+1}\vp(a_{n+1}a_0,\ldots,a_n).\nonumber
\end{align}
With these definitions at hand, we define the cyclic cohomology $H_\lambda(\Ac)$ of the algebra $\Ac$ as the homology of the subcomplex
\begin{equation}
C_\lambda^*(\Ac)=\bigoplus_{n\geq 0} C^n_\lambda(\Ac,\Ac^\ast),
\end{equation}
where
\begin{equation}
C^n_\lambda(\Ac,\Ac^\ast):=\{\phi\in C^n(\Ac,\Ac^\ast)\,|\, \phi(a_0,\ldots,a_n)=(-1)^n\phi(a_n,a_0,\ldots,a_{n-1})\}.
\end{equation}

Equivalently, the cyclic cohomology $HC(\Ac)$ of an algebra $\Ac$ can also be defined as the cyclic cohomology of the cocyclic module associated to
\begin{equation}
C^*(\Ac) = \bigoplus_{n\geq 0} C^n(\Ac),\qquad C^n(\Ac):=\Hom(\Ac^{\ot\,n+1},k),
\end{equation}
by its own cofaces, codegeneracies and cyclic group actions.  The
coface maps $d_k:C^n(\Ac)\lra C^{n+1}(\Ac)$ are defined for
$0\leq k\leq n+1$ as
\begin{align*}
 d_k\vp(a^0,\ldots, a^{n+1}) = 
  \begin{cases}
    \vp(a^0,\ldots, a^ka^{k+1}, \ldots, a^{n+1}) & \text{ if } 0\leq k\leq n,\\
    \vp(a^{n+1}a^0, a^1, \ldots, a^n) & \text{ if } k=n+1.
  \end{cases}
\end{align*}
The codegenerecy maps $s_j:C^n(\Ac)\lra C^{n-1}(\Ac)$ are defined for
$0\leq j\leq n-1$ as
\begin{align*}
s_j\vp(a^0,\ldots, a^{n-1}) = \vp(a^0,\ldots, a^j, 1, a^{j+1}, \ldots, a^{n+1}),
\end{align*}
and finally the cyclic operators $t_n:C^n(\Ac)\lra C^n(\Ac)$ as
\begin{align*}
t_n\vp(a^0,\ldots, a^n) = \vp(a^n, a^0,\ldots, a^{n-1}).
\end{align*}
The cyclic cohomology of $\mathcal{A}$ is defined to be the total
cohomology of the associated first quadrant bicomplex
$\left(CC(\Ac),b,B\right)$ where
\begin{equation*}
CC^{p,q}(\Ac):= \begin{cases}
                           C^{q-p}(\Ac) & \text{if}\,\,q\geq p \geq 0, \\
                           0 & \text{if}\,\, p>q,
                         \end{cases}
\end{equation*}
with the algebra Hochschild coboundary operator
$b: CC^{p,q}(\Ac)\lra CC^{p,q+1}(\Ac)$ which is given by
\begin{equation*}
b:=\sum_{i=0}^{q+1}(-1)^i d_i,
\end{equation*}
and the Connes boundary operator
$B: CC^{p,q}(\Ac)\lra CC^{p-1,q}(\Ac)$ which is defined as
\begin{equation*}
 B:=\left(\sum_{i=0}^{p}(-1)^{pi}t^{i}_{p}\right)(1+(-1)^pt_p) s_{p}.
\end{equation*}
We recall that $H_\lambda^*(\Ac)\cong HC^*(\Ac)$, since we assume
throughout that the ground field $k$ is of characteristic $0$.

\subsection{Twisted cyclic cohomology}\label{TwistedCyclic}

We next briefly recall from \cite{KustMurpTuse03} the
twisted cyclic cohomology of an algebra $\Ac$ by an
automorphism $\s:\Ac\lra \Ac$. Let $C^n(\Ac)$ be the set of all linear maps $\Ac^{\ot\,n+1}\lra k$. Then, the complex
\begin{equation}\label{twisted-complex}
C_\s^*(\Ac) = \bigoplus_{n\geq 0} C^n_\s(\Ac),
\end{equation}
where 
\begin{equation}
C^n_\s(\Ac) = \{\phi\in C^n(\Ac)\mid \phi(a_0,\ldots,a_n) = (-1)^n\phi(\s(a_n),a_0,\ldots,a_{n-1})\},
\end{equation}
is closed under the twisted Hochschild differential,
\begin{align}\label{twisted-Hoch}
b\phi(a_0,\ldots,a_{n+1}) 
= & \phi(a_0a_1,\ldots,a_{n+1})\nonumber\\
  & + \sum_{k=1}^n(-1)^k\vp(a_0,\ldots,a_ka_{k+1}\ldots,a_{n+1}) \\
  & + (-1)^{n+1}\vp(\s(a_{n+1})a_0,\ldots,a_n).\nonumber
\end{align}
Then the homology of the complex \eqref{twisted-complex} with the differential map \eqref{twisted-Hoch} is called the $\s$-twisted cyclic cohomology of the algebra $\Ac$.

Equivalently, the $\s$-twisted cyclic cohomology $HC_\s^*(\Ac)$ of the algebra $\Ac$ is computed by the cocyclic object 
\begin{equation}
C_\s^n(\Ac) = \{\phi\in C^n(\Ac)\mid \phi(\s(a_0),\ldots,\s(a_n)) = \phi(a_0,\ldots, a_n)\}
\end{equation}
given by the coface maps $d_k:C_\s^n(\Ac)\lra C_\s^{n+1}(\Ac)$ for $0\leq k\leq n+1$,
\begin{align*}
d_k\vp(a^0,\ldots, a^{n+1}) =
  \begin{cases}
    \vp(a^0,\ldots, a^ka^{k+1}, \ldots, a^{n+1}) & \text{ if } 0\leq k\leq n,\\
    \vp(\s(a^{n+1})a^0, a^1, \ldots, a^n) & \text{ if } k=n+1,
  \end{cases}
\end{align*}
the codegeneracy maps $s_j:C_\s^n(\Ac)\lra C_\s^{n-1}(\Ac)$
for $0\leq j\leq n-1$ as 
\begin{align*}
s_j\vp(a^0,\ldots, a^{n-1}) = \vp(a^0,\ldots, a^j, 1, a^{j+1}, \ldots, a^{n+1}),
\end{align*}
and finally the cyclic operators $t_n:C_\s^n(\Ac)\lra C_\s^n(\Ac)$
\begin{align*}
t_n\vp(a^0,\ldots, a^n) = \vp(\s(a^n), a^0,\ldots, a^{n-1}).
\end{align*}

\subsection{Hopf-cyclic cohomology}

In this subsection we recall the Hopf-cyclic cohomology for Hopf algebras from \cite[Sect. 3\&4]{ConnMosc}, see also \cite[Sect. 2]{Crai02}.

Let $\Hc$ be a Hopf algebra with a modular pair $(\d,\s)$ in involution (MPI). 
In other words, $\d:\Hc\lra k$ is a character, and
$\s\in\Hc$ a group-like satisfying the modularity condition
\begin{equation}
\d(\s) = 1 \quad\text{ and }\quad S_\d^2 = \Ad_\s,\quad\text{ where }\quad S_\d(h) := \d(h\ps{1})S(h\ps{2}),
\end{equation}
for all $h\in H$. Then the Hopf-cyclic cohomology $HC^\ast(\Hc,\d,\s)$ of $\Hc$, relative to the pair $(\d,\s)$ is defined to be the cyclic cohomology of the cocyclic module \cite{ConnMosc98}
\begin{equation*}
C^\ast(\Hc,\d,\s)=\bigoplus_{n\geq 0}C^n(\Hc,\d,\s),\qquad C^n(\Hc,\d,\s):= \Hc^{\ot\, n}.
\end{equation*}
The coface operators $d_i: C^n(\Hc,\d,\s)\ra C^{n+1}(\Hc,\d,\s)$ are
defined for $0\le i\le n+1$ as
\begin{align}\label{aux-Hopf-cyclic-coface}
d_i(h^1\odots h^n) = 
  \begin{cases}
    1\ot h^1\odots h^n, & \text{ if } i=0,\\
    h^1\odots h^i\ps{1}\ot h^i\ps{2}\odots h^n, & \text{ if } 0\leq i\leq n,\\
    h^1\odots h^n\ot \s, & \text{ if } i=n+1.
  \end{cases}
\end{align}
We define the codegeneracy operators
$s_j: C^n(\Hc,\d,\s)\ra C^{n-1}(\Hc,\d,\s)$ for $0\le j\le n-1$ as
\begin{align}\label{aux-Hopf-cyclic-codegeneracy}
s_j (h^1\odots h^n)= h^1\odots \ve(h^{j+1})\odots h^n,
\end{align}
and the cyclic operators $t_n: C^n(\Hc,\d,\s)\ra C^n(\Hc,\d,\s)$ as
\begin{align}\label{aux-Hopf-cyclic-cyclic}
t_n(h^1\odots h^n)=S_\d(h^1)\cdot(h^2\odots h^n\ot \s).
\end{align}

\subsection{Connes-Moscovici characteristic map}

Let us next recall the Connes-Moscovici characteristic homomorphism, \cite{ConnMosc98,ConnMosc}. Let $\Ac$ be a $\Hc$-module algebra, that is, for all
$h\in \Hc$ and for all $a,b\in\Ac$,
\begin{equation}
h(ab) = h\ps{1}(a)h\ps{2}(b),\qquad h(1)=\ve(h)1.
\end{equation}
Then a linear form $\tau:\Ac\lra k$ is called a $\s$-trace if
\begin{equation}
\tau(ab)=\tau(b\s(a)),
\end{equation}
and the $\s$-trace $\tau:\Ac\lra k$ is called $\d$-invariant if
\begin{equation}
\tau(h(a))=\d(h)\tau(a),
\end{equation}
for all $h\in\Hc$ and $a,b\in \Ac$.

Now let $\Hc$ be a Hopf algebra with a MPI $(\d,\s)$, and $\Ac$ an $\Hc$-module algebra equipped with a $\d$-invariant $\s$-trace. It follows then that the morphisms $\chi_\tau\colon C^n(\Hc,\d,\s) \lra C^n(\Ac)$,
\begin{equation}\label{aux-char-homom-H-A}
\chi_\tau(h^1 \odots h^n)(a^0,\ldots,a^n) := \tau(a^0h^1(a^1)\ldots h^n(a^n)),
\end{equation}
induce a characteristic homomorphism on the cohomology $\chi_\tau\colon HC(\Hc,\d,\s)\lra HC(\Ac)$.

\subsection{Hopf-cyclic cohomology of module algebras}
We now recall from \cite{HajaKhalRangSomm04-II} the Hopf-cyclic cohomology theory for the module algebra symmetry. In order to define the coefficient spaces, we first note that a modular pair in involution is an example of a one dimensional stable anti-Yetter Drinfeld (SAYD) module, \cite{HajaKhalRangSomm04-I}. In general, a right module - left comodule $V$ over a Hopf algebra $\Hc$ is called a right-left SAYD module over $\Hc$ if
\begin{equation*}
\nb(h\cdot v) = S(h\ps{3})v\ns{-1}h\ps{1}\ot v\ns{0},\qquad v\ns{0}\cdot v\ns{-1}=v
\end{equation*}
for any $v\in V$ and any $h\in\Hc$. Here $\nb:V\lra\Hc\ot V$ given by $v\mapsto v\ns{-1}\ot v\ns{0}$ refers to the left $\Hc$-coaction on $V$.

Let $\Ac$ be an $\Hc$-module algebra and $V$ a SAYD module over $\Hc$. We recall from \cite{HajaKhalRangSomm04-II,Rang08} that the graded space
\begin{equation*}
C_\Hc^*(\Ac,V) = \bigoplus_{n\geq}C_\Hc^n(\Ac,V),\qquad C_\Hc^n(\Ac,V):=\Hom_\Hc(V\ot\Ac^{\ot\,n+1},k)
\end{equation*}
becomes a cocyclic module via the coface maps
$\p_i:C_\Hc^n(\Ac,V)\lra C_\Hc^{n+1}(\Ac,V)$, defined for
$0\leq i\leq n+1$ 
\begin{align*}
  \p_i\vp(v, a^0,\ldots, a^{n+1}) = 
  \begin{cases}
    \vp(v,a^0,\ldots, a^ia^{i+1}, \ldots, a^{n+1}), 
            & \text{ if } 0\leq i\leq n,\\
    \vp(v\ns{0},S^{-1}(v\ns{-1})(a^{n+1})a^0, a^1, \ldots, a^n), 
            & \text{ if } i = n+1,
  \end{cases}
\end{align*}
the codegenerecy maps
$s_j:C_\Hc^n(\Ac,V)\lra C_\Hc^{n-1}(\Ac,V)$, defined for $0\leq j\leq n-1$ by
\begin{align*}
 s_j\vp(v,a^0,\ldots, a^{n-1})  = \vp(v,a^0,\ldots, a^j, 1, a^{j+1}, \ldots, a^{n-1}),
\end{align*}
and the cyclic operators $t_n:C_\Hc^n(\Ac,V)\lra C_\Hc^n(\Ac,V)$,
\begin{align*}
t_n \vp(v,a^0,\ldots, a^n) 
= \vp(v\ns{0},S^{-1}(v\ns{-1})(a^n), a^0,\ldots, a^{n-1}).
\end{align*}
The cyclic homology of this cocyclic module is called the Hopf-cyclic cohomology of the $\Hc$-module algebra $\Ac$ with coefficients, and is denoted by $HC_\Hc^\ast(\Ac,V)$. We note from \cite{HajaKhalRangSomm04-II} that if $\s\in{\rm Aut}(\Ac)$, then with the Hopf algebra $\Hc=k[\s,\s^{-1}]$ of Laurent polynomials and $V={}^{\s^{-1}} k$ we recover the twisted cyclic cohomology.

\subsection{Hopf-cyclic cohomology of module coalgebras}\label{CocyclicObject}
Let us next recall the Hopf-cyclic cohomology of module coalgebras with SAYD coefficients. Let $\Cc$ be a left $\Hc$-module coalgebra.  That is, $\Hc$ acts on $\Cc$ such that
\begin{equation}
\D(h\cdot c) = h\ps{1}\cdot c\ps{1} \ot h\ps{2}\cdot c\ps{2}, \qquad \ve(h\cdot c) = \ve(h)\ve(c),
\end{equation}
for any $h\in \Hc$, and any $c\in \Cc$. Let also $V$ be a right-left SAYD module over $\Hc$. Then the Hopf-cyclic cohomology of $\Cc$ under the symmetry of 
$\Hc$ is given by the cocyclic module of the coface operators $\p_i: C^n_\Hc(\Cc,V)\lra C^{n+1}_\Hc(\Cc,V)$ defined for $0\le i\le n+1$ by
\begin{align}\label{aux-coface-C}
\p_i(v\ot_\Hc c^0\odots c^n) = 
  \begin{cases}
  v\ot_\Hc c^0\odots c^i\ps{1}\ot c^i\ps{2}\odots c^n, 
         & \text{ if } 0\leq i\leq n,\\
  v\ns{0}\ot_\Hc c^0\ps{2} \ot c^1\odots c^n\ot v\ns{-1}\cdot c^0\ps{1},
         & \text{ if } i=n+1,
  \end{cases}
\end{align}
the codegeneracy operators
$\s_j: C^n_\Hc(\Cc,V)\lra C^{n-1}_\Hc(\Cc,V)$ for $0\le j\le n-1$
\begin{align}\label{aux-codeg-C}
\s_j (v\ot_\Hc c^0\odots c^n)= v\ot_\Hc c^0\odots \ve(c^{j+1})\odots c^n,
\end{align}
and the cocyclic operators $\tau_n: C^n_\Hc(\Cc,V)\lra C^n_\Hc(\Cc,V)$
\begin{align}\label{aux-cyclic-C}
\tau_n(v\ot_\Hc c^0\odots c^n)=v\ns{0}\ot_\Hc c^1\odots v\ns{-1}\cdot c^0.
\end{align}
The cyclic homology of this cocyclic module is denoted by $HC_\Hc^\ast(\Cc,V)$. In particular, if $\Cc=\Hc$ which is considered as a left $\Hc$-module coalgebra by the left regular action of $\Hc$ on itself, the Hopf-cyclic cohomology with coefficients of the Hopf algebra $\Hc$ is denoted by $HC^\ast(\Hc,V)$. If, furthermore, $V={}^\s k_\d$ the one dimensional SAYD module by a MPI $(\d,\s)$ of $\Hc$, the cocyclic structure given by \eqref{aux-coface-C}, \eqref{aux-codeg-C} and \eqref{aux-cyclic-C} reduces to the one given by \eqref{aux-Hopf-cyclic-coface}, \eqref{aux-Hopf-cyclic-codegeneracy} and \eqref{aux-Hopf-cyclic-cyclic}, \cite{HajaKhalRangSomm04-II}.

\subsection{The characteristic map and untwisting}\label{Untwisting}
Let us recall from \cite[Thm. 2.8]{Kayg08} and
\cite[Prop. 2.3]{Rang08} that if $\Ac$ is a left $\Hc$-module algebra and $V$ a right-left SAYD module over $\Hc$, then there is a cup
product
\begin{equation}\label{cup}
\cup:HC^p_\Hc(\Ac,V)\ot HC^q(\Hc,V) \lra HC^{p+q}(\Ac).
\end{equation}
On the level of Hochschild cohomology, it is given by the formula
\begin{equation*}
(\vp\cup(v\ot h^1\odots h^p))(a^0,\ldots, a^{p+q}) = \vp(v,a^0h^1(a^1)\ldots h^p(a^p), a^{p+1},\ldots, a^{p+q}),
\end{equation*}
and, following \cite{KhalRang05-II,Rang08}, in the level of cyclic cohomology by
\begin{align*}
(\vp\cup \widetilde{h})(a^0,\ldots, a^{p+q}) = \sum_{\mu\in Sh(q,p)} (-1)^{\mu}\p_{\wbar{\mu}(q)}\ldots \p_{\wbar{\mu}(1)}\vp(\p_{\wbar{\mu}(q+p)}\ldots \p_{\wbar{\mu}(q+1)}\widetilde{h}(a^0,\ldots, a^{p+q})),
\end{align*}
where $\widetilde{h} = v\ot_\Hc h^0\odots h^p$, $\wbar{\mu}(\ell) = \mu(\ell) - 1$, and $Sh(q,p)$ denotes the set of all $(p,q)$-shuffles. We note also that, for a Hopf algebra $\Hc$ with a MPI $(\d,\s)$, the cup product by a 0-cocycle $\tau \in HC^0(\Hc,{}^\s k_\d)$ induces the Connes-Moscovici characteristic homomorphism \eqref{aux-char-homom-H-A}.

We will use the cup product \eqref{cup} to untwist the twisted cyclic cohomology. To this end, we first note that the Hopf-cyclic cohomology of
the Hopf algebra of Laurent polynomials $k[\sigma,\sigma^{-1}]$ with coefficients in the SAYD module corresponding to the MPI
$(\varepsilon,\sigma^{-1})$, is concentrated in degree 1.  More precisely,
\begin{equation*}
  HC^1(k[\sigma,\sigma^{-1}],{}^{\s^{-1}}k)=\langle \one\ot (1-\s^{-1}) \rangle,
\end{equation*}
Then specializing \eqref{cup} to
\begin{equation*}
  \cup \colon HC^p_\sigma(\mathcal{A}) \otimes
HC^1(k[\sigma,\sigma^{-1}],{}^{\sigma^{-1}}k) \lra HC^{p+1}(\mathcal{A}),
\end{equation*}
we get a characteristic map
\begin{equation}\label{aux-char-map-untwist}
\chi:HC^n_\s(\Ac) \lra HC^{n+1}(\Ac),
\end{equation}
which is given in the level of Hochschild cohomology by
\begin{equation}\label{aux-char-map-untwist-hoch-level}
\chi(\vp)(a^0,\ldots,a^{n+1})=\vp(a^0(1-\s^{-1})(a^1),a^2,\ldots,a^{n+1}),
\end{equation}
and in the level of cyclic cohomology by
\begin{align*}
\chi(\vp)(a^0,\ldots,a^{n+1}) = \sum_{\mu\in Sh(1,n)} (-1)^{\mu}d_{\wbar{\mu}(1)}\vp(\p_{\wbar{\mu}(n+1)}\ldots \p_{\wbar{\mu}(2)}(1-\s^{-1})(a^0,\ldots, a^{p+q})).
\end{align*}

We would like to note here that the untwisting phenomenon via a cup
product explains in part the dimension drop phenomenon observed in
\cite{FengTsyg91}, and also Goodman and Krähmer's result
\cite[Thm. 1.1]{GoodKrah14} that the smash product of a twisted
Calabi-Yau algebra of dimension $d$ with the Laurent polynomial ring
is an untwisted Calabi-Yau algebra of dimension $d+1$.

\subsection{The characteristic map for compact
  groups}\label{ClassicalVanEst}

We conclude this section by investigating the characteristic
homomorphism \eqref{aux-char-homom-H-A} following \cite{ConnMosc98}
(see also \cite{HochKostRose62}) in the case of $\Hc=U(\Fg)$ and
$\Ac=\Oc(G)$ where $G$ is one of the (unimodular) groups $SL(N), SO(N)$ or $Sp(N)$
where we have a non-trivial invariant Haar functional, and $\Fg$ the
Lie algebra of $G$. In these cases, $(\ve,1)$ is a MPI for the Hopf
algebra $U(\Fg)$, and $\Oc(G)$ is a left $U(\Fg)$-module algebra via
$u(f)(x) := f(x\lt u)$ induced from the action of $\Fg$ on $G$ for any
$u\in U(\Fg)$, any $f\in \Oc(G)$, and any $x\in G$.

Let $\mu$ be the Haar measure on $G$. Then the functional $h\colon \Oc(G)\lra k$ defined by $h(f):=\int_G f(x)d\mu(x)$ form an invariant trace for the (commutative) Hopf algebra $\Oc(G)$. Indeed, for any $X\in \Fg$, and any $f \in \Oc(G)$,
\begin{align}
h(X\rt f) 
= &\int_G (X\rt f)(x)d\mu(x) \nonumber\\
= & \int_G \dt f(\exp(tX)\,x\,\exp(tX))d\mu(x) \nonumber\\
= & \int_G \dt f(y)d\mu(\exp(tX)\,y\,\exp(tX))\\
= & \int_G \dt f(y)d\mu(y)\nonumber\\
= & 0 = \d(X)h(f),\nonumber
\end{align}
where on the third equality we use the invariance of the Haar measure. We also note that by the unimodularity of the group $G$, the trace of the adjoint representation of $\Fg$ on itself vanishes. As a result, we have $ \chi:HC^\ast(U(\Fg), k_\d) \lra HC^\ast(\Oc(G))$ defined as
\begin{align}\label{aux-char-homom}
\chi(u^1,\ldots,u^n)(f^0,\ldots,f^n):=\int_G f^0(x)(u^1(f^1))(x)\ldots (u^n(f^n))(x)d\mu(x),
\end{align}
where $u(f)\in \Oc(G)$, for an arbitrary $u\in U(\Fg)$ and $f\in \Oc(G)$, denotes the left coregular action.
We also recall from \cite[Thm. 15]{ConnMosc} that 
\begin{equation*} 
HP^\ast(U(\Fg), k_\d) \cong \displaystyle\bigoplus_{n =\ast\,{\rm
    mod}\,2} \,H_n(\Fg,k_\d) 
\end{equation*}
via the anti-symmetrization
$ant: k_\d\,\ot\,\bigwedge^\ast\Fg \lra C^\ast(U(\Fg), k_\d)$, and from
\cite[Thm. 46]{Conn85} that
\begin{equation*}
HP^\ast(\Oc(G))\cong \displaystyle\bigoplus_{n =\ast\,{\rm mod}\,2}
\,H^{\rm dR}_n(G) 
\end{equation*}
via a map $\vp\mapsto C$ given by
\begin{equation*}
\langle C, f^0df^1\wdots df^n\rangle = \sum_{\s\in S_n} (-1)^\s
\vp(f^0,f^{\s(1)},\ldots, f^{\s(n)}).
\end{equation*}
Here $H^{\rm dR}_\ast(G)$ refers to the de~Rham homology of $G$. In
the reverse direction, from \cite[Thm. 3.2.14]{Connes-book} we have
$\Phi\colon H^{\rm dR}_\ast(G)\lra HC^\ast(\Oc(G))$ given by
\[ \Phi(C)(f^0,f^1,\ldots, f^n) = \langle C, f^0df^1\wdots
df^n\rangle. \]
Hence, following \cite[Lemma 8\,\& 9]{ConnMosc98}, we arrive at the commutative diagram
\begin{equation}
\xymatrix{
HP^\ast(U(\Fg), k_\d) \ar[r]^\chi & HP^\ast(\Oc(G)) \\
\displaystyle\bigoplus_{n =\ast\,{\rm mod}\,2} \,H_n(\Fg,k_\d) \ar[r]\ar[u]^{ant} & \displaystyle\bigoplus_{n =\ast\,{\rm mod}\,2} \,H^{\rm dR}_n(G) \ar[u]_\Phi
}
\end{equation}
which is the periodic version of \eqref{aux-char-homom} up to Poincaré duality.

\section{Quantum characteristic map}
In this section we will define a quantum analogue of the characteristic homomorphism for compact quantum group algebras. To this end we will first recall the quantum enveloping algebras, and their Hopf-cyclic cohomology, as well as the compact quantum group algebras from \cite{KlimSchm-book}. Then using the modular property of the Haar functional we construct a characteristic homomorphism similar to that of Connes and Moscovici \cite{ConnMosc98}. 

\subsection{Quantum enveloping algebras (QUE algebras)}\label{subsect-Uq-g}

Following \cite[Subsect. 6.1.2]{KlimSchm-book}, let $\Fg$ be a finite dimensional complex semi-simple Lie algebra, 
$A=[a_{ij}]$ the Cartan matrix of $\Fg$, and $d_i \in \{1,2,3\}$ for $1\leq i\leq \ell$ so that $DA=[d_ia_{ij}]$ is the symmetrized Cartan matrix. Let also $q$ be a fixed nonzero complex number such that $q_i^2\neq 1$, where $q_i:= q^{d_i}$.

Then the quantum enveloping algebra $U_q(\Fg)$ is the Hopf algebra with $4\ell$ generators $E_i,F_i,K_i,K_i^{-1}$, $1\leq i\leq \ell$, and the relations
\begin{align*}
& K_iK_j = K_jK_i,\qquad K_iK_i^{-1}=K_i^{-1}K_i=1,\\
& K_iE_jK_i^{-1}=q_i^{a_{ij}}E_j,\qquad K_iF_jK_i^{-1}=q_i^{-a_{ij}}F_j,\\
& E_iF_j-F_jE_i=\d_{ij}\frac{K_i-K_i^{-1}}{q_i-q_i^{-1}},\\
& \sum_{r=0}^{1-a_{ij}}(-1)^r\left[\begin{array}{c}
                                              1-a_{ij} \\
                                              r
                                            \end{array}
\right]_{q_i}E_i^{1-a_{ij}-r}E_jE_i^r=0, \quad i\neq j,\\
& \sum_{r=0}^{1-a_{ij}}(-1)^r\left[\begin{array}{c}
                                              1-a_{ij} \\
                                              r
                                            \end{array}
\right]_{q_i}F_i^{1-a_{ij}-r}F_jF_i^r=0, \quad i\neq j,
\end{align*}
where
\begin{equation*}
\left[\begin{array}{c}
        n \\
        r
      \end{array}
\right]_q = \frac{(n)_q\,!}{(r)_q\,!\,\,(n-r)_q\,!},\qquad (n)_q:=\frac{q^n-q^{-n}}{q-q^{-1}}.
\end{equation*}
The rest of the Hopf algebra structure of $U_q(\Fg)$ is given by
\begin{align}\label{aux-Uq-mod-alg}
\begin{split}
& \D(K_i)=K_i\ot K_i,\quad \D(K_i^{-1})=K_i^{-1}\ot K_i^{-1} \\
& \D(E_i)=E_i\ot K_i + 1\ot E_i,\quad \D(F_j)=F_j\ot 1 + K_j^{-1}\ot F_j \\
& \ve(K_i)=1,\quad \ve(E_i)=\ve(F_i)=0\\
& S(K_i)=K_i^{-1},\quad S(E_i)=-E_iK_i^{-1},\quad S(F_i)=-K_iF_i.
\end{split}
\end{align}

\subsection{Cohomology of QUE algebras}\label{CohomologyOfQUEAs}
Let us recall the Hochschild cohomology of the quantized enveloping
algebras $U_q(\Fg)$ from \cite{KaygSutl14}.  However, we develop here
a different strategy than \emph{op.cit.}

A modular pair in involution (MPI) for the Hopf algebra $U_q(\Fg)$ is
given by \cite[Prop. 6.1.6]{KlimSchm-book}. Let
$K_\lambda:=K_1^{n_1}\ldots K_\ell^{n_\ell}$ for any
$\lambda=\sum_in_i\a_i$, where $n_i\in \Zb$. Then, $\rho \in \Fh^\ast$
being the half-sum of the positive roots of $\Fg$, by
\cite[Prop. 6.1.6]{KlimSchm-book} we have
\begin{equation}\label{aux-gr-like-q}
S^2(a) = K_{2\rho}aK^{-1}_{2\rho}
\end{equation}
for all $a\in U_q(\Fg)$. Thus, $(\ve,K_{2\rho})$ is a MPI for the Hopf
algebra $U_q(\Fg)$. We shall use the notation $\s:=K_{2\rho}$. In view
of the arguments in Subsection \ref{subsect-cohom-coext}, and
following \cite{KaygSutl14}, for
\begin{equation*}
U_q(\Fb_+)={\rm Span}\left\{E^{r_1}_1\ldots E_\ell^{r_\ell}K^{q_1}_1\ldots K_\ell^{q_\ell}\,|\,r_1,\ldots,r_\ell\geq 0,\,q_1,\ldots, q_\ell \in \Zb\right\},
\end{equation*}
we consider the coextension $\pi:U_q(\Fg)\to U_q(\Fb_+)$ which is defined as
\begin{equation*}
  \pi(E_1^{r_1}\cdots E_\ell^{r_\ell}K_1^{q_1}\cdots K_\ell^{q_\ell}F_1^{s_1}\cdots F_\ell^{s_\ell})= 
  \begin{cases}
    E_1^{r_1}\cdots E_\ell^{r_\ell}K_1^{q_1}\cdots K_\ell^{q_\ell} & \text{ if } r_1+\cdots+r_\ell=0,\\
    0 & \text{ otherwise}.
  \end{cases}
\end{equation*}
Because we have a Poincaré-Birkhoff-Witt basis for $U_q(\Fg)$, it
is coflat over the coalgebra $U_q(\Fb_+)$.  Thus, by
\cite[Prop. 4.8]{KaygSutl14} and \cite[Lemma 5.1]{Crai02},
\begin{equation}
  HH^n(U_q(\Fg),{}^{\s} k)=
  \begin{cases}
    k^{\oplus \,2^\ell} & \text{ if } n=\ell \\
    0 & \text{ if } n\neq \ell.
  \end{cases}
\end{equation}
In particular, for $\Fg=s\ell_2$, we calculate the same classes as
\cite[Prop. 5.9]{Crai02}. Namely,
\begin{equation}\label{aux-Hopf-cyclic-class-Uq-sl2}
  HH^n(U_q(s\ell_2),{}^{\s} k) = 
  \begin{cases}
    \langle E, KF\rangle & \text{ if } n=1 \\
    0 & \text{ if } n\neq 1.
  \end{cases}
\end{equation}

We finally note that along the way to compute the Hochschild (co)homology of $U_q(\Fg)$, regarded as an algebra, the $\Tor$-groups 
$\Tor_*^{U_q(\mathfrak{g})}(k,k)$ and the $\Ext$-groups
$\Ext_{U_q(\mathfrak{g})}^*(k,k)$ are obtained in \cite{Drupieski:2013}.

\subsection{Compact quantum group algebras (CQG algebras)}\label{CQGAlgebras}

In this subsection we will construct a characteristic map
$HC^p(U_q(\Fg),{}^\s k) \lra HC^{p+1}(\Oc(G_q))$.
In order to do this, we will use the existence of a unique Haar state on the
coordinate algebras, as well as their pairing with the QUE algebras.

We begin with the definition of the coordinate algebras of the quantum
groups from \cite[Sect. 11.3]{KlimSchm-book}.

\begin{definition}
A Hopf $\ast$-algebra $\Hc$ is called a compact quantum group (CQG) algebra if $\Hc$ is the linear span of all matrix elements of finite dimensional unitary corepresentations of $\Hc$. A compact matrix quantum group (CMQG) algebra is a CQG algebra which is generated, as an algebra, by finitely many elements.
\end{definition}
Among examples of CMQG algebras are the Hopf $\ast$-algebras $\Oc(U_q(N))$, $\Oc(SU_q(N))$, $\Oc(O_q(N;\Rb))$, $\Oc(SO_q(N;\Rb))$ and $\Oc(Sp_q(N))$, see \cite[Ex. 11.7]{KlimSchm-book}. For any compact group $G$, the Hopf algebra $R(G)$ of representative functions is a CQG algebra. Also, if $\mu$ is the Haar measure on such a group $G$ then $h:R(G)\lra k$ given by $h(f):=\int_G f(x)d\mu(x)$ is the corresponding Haar functional.

\begin{theorem}\label{MainResult1}
  If $\Ac = \Oc(G_q)$ is a CQG
  algebra, then there is a characteristic map of the form
  \begin{align}\label{aux-char-homom-q}
    & \chi_q: HC^\ast(U_q(\Fg), {}^{\s} k) \lra HC_{\s^{-1}}^\ast(\Oc(G_q)),\\
    & \chi_q (y^1,\ldots,y^n)(f^0,\ldots,f^n):=h(f^0y^1(f^1)\ldots y^n(f^n)).\nonumber
  \end{align}
\end{theorem}

\begin{proof}
  Every CQG algebra possesses a unique (left and right invariant) Haar
  functional due to their cosemisimplicity
  \cite[Thm. 11.13]{KlimSchm-book}. In view of
  \cite[Eq. 11(36)]{KlimSchm-book}, it follows from
  \cite[Prop. 11.34]{KlimSchm-book} that the Haar functional
  $h:\Ac\lra k$ on a CQG algebra of the form $\Ac = \Oc(G_q)$ has the
  crucial property that
  \begin{equation}\label{modularity}
    h(ab)=h(b(\s\rt a\lt \s))
  \end{equation}
  for $\s=K_{2\rho}\in U_q(\Fg)$ with the left ad the right coregular
  actions. We will observe the compatibility of the map
  \eqref{aux-char-homom-q} with the Hopf-cyclic coface operators
  \eqref{aux-Hopf-cyclic-coface}, codegeneracies
  \eqref{aux-Hopf-cyclic-codegeneracy}, and the cyclic operator
  \eqref{aux-Hopf-cyclic-cyclic}. Accordingly, we first show that
\begin{align*}
\chi_q(d_0(y^1,\ldots,y^n))(f^0,\ldots,f^{n+1}) 
= & \chi_q(1,y^1,\ldots,y^n)(f^0,\ldots,f^{n+1})\\
= & h(f^0f^1y^1(f^2)\ldots y^n(f^{n+1})) \\
= & d_0\chi_q(y^1,\ldots,y^n)(f^0,\ldots,f^{n+1}).
\end{align*}
Next we observe for $1\leq i \leq n$ that
\begin{align*}
\chi_q(d_i(y^1,\ldots,y^n))(f^0,\ldots,f^{n+1}) 
= & \chi_q(y^1,\ldots,\D(y^i),\ldots,y^n)(f^0,\ldots,f^{n+1}) \\
= & h(f^0y^1(f^1)\ldots y^i(f^if^{i+1})\ldots y^n(f^{n+1}))\\
= & d_i\chi_q(y^1,\ldots,y^n)(f^0,\ldots,f^{n+1}).
\end{align*}
As for the last coface map we have
\begin{align*}
\chi_q(d_{n+1}(y^1,\ldots,y^n))(f^0,\ldots,f^{n+1}) 
= & \chi_q(y^1,\ldots, y^n,\s)(f^0,\ldots,f^{n+1}) \\
= & h(f^0y^1(f^1)\ldots y^n(f^n)\s(f^{n+1}))\\
= & h((f^{n+1}\lt \s^{-1})f^0y^1(f^1)\ldots y^n(f^n))\\
= & \chi_q(y^1,\ldots,y^n)((f^{n+1}\lt \s^{-1})f^0,\ldots,f^n)\\
= & d_{n+1}\chi_q(y^1,\ldots,y^n)(f^0,\ldots,f^{n+1}).
\end{align*}
We proceed to the codegeneracies. We have,
\begin{align*}
\chi_q(s_j(y^1,\ldots,y^n))(f^0,\ldots,f^{n-1})
= & \chi_q(y^1,\ldots,\ve(y^j),\ldots,y^n)(f^0,\ldots,f^{n-1})\\
= & \ve(y^j)h(f^0y^1(f^1)\ldots y^{j-1}(f^{j-1})y^{j+1}(f^{j+1})\ldots y^n(f^{n-1})) \\
= & \chi_q(y^1,\ldots,y^n)(f^0,\ldots,f^j,1,f^{j+1},\ldots f^{n-1})\\
= & s_j\chi_q(y^1,\ldots,y^n)(f^0,\ldots,f^{n-1}).
\end{align*}
Finally, we consider the compatibility with the cyclic operator. We have
\begin{align*}
\chi_q(t(y^1,\ldots,y^n))(f^0,\ldots,f^{n+1})
= & \chi_q(S(y^1)(y^2,\ldots,y^n,\s))(f^0,\ldots,f^n)\\
= & h(f^0S(y^1)(y^2,\ldots,y^n,\s)(f^1,\ldots, f^n))\\
= & h(y^1\ps{1}(f^0S(y^1\ps{2})(y^2,\ldots,y^n,\s)(f^1,\ldots, f^n)))\\
= & h(y^1(f^0)y^2(f^1)\ldots y^n(f^{n-1})\s(f^n))\\
= & h((f^n\lt \s^{-1})y^1(f^0)y^2(f^1)\ldots y^n(f^{n-1}))\\
= & \chi_q(y^1,\ldots,y^n)((f^{n+1}\lt \s^{-1}), f^0,\ldots,f^n)\\
= & t\chi_q(y^1,\ldots,y^n)(f^0,\ldots,f^{n+1}).
\end{align*}
As a result, for a compact group $G$ with Lie algebra $\Fg$, the
morphism $\chi_q$ defined on the chain level by \eqref{aux-char-homom-q} induces a morphism $HC^\ast(U_q(\Fg), {}^{\s} k) \lra HC_{\s^{-1}}^\ast(\Oc(G_q))$ in the level of
cohomology.
\end{proof}

\begin{remark}
  We would like remark that the modularity~\eqref{modularity} of the
  Haar functional is not given by a module algebra action of
  $U_q(\Fg)$ on $\C{O}(G_q)$.  However, it is observed in \cite[Thm. 1,
  Thm. 2]{KRT} that it can be viewed as the module algebra action of the modular
  square of $U_q(\Fg)$, see \cite[Ex. 3.14]{KRT1}. Then the same Haar
  functional induces a Connes-Moscovici characteristic map whose
  target is now the ordinary cyclic cohomology of $\C{O}(G_q)$, \cite[Thm. 2]{KRT}, 
  which is observed to be zero for $G=SU(2)$, \cite[Thm. 9]{KRT}. 
  On the other hand, the characteristic map \eqref{aux-char-homom-q} 
  is also given by \cite[Thm. 8.2]{KRT1}.
\end{remark}

From Subsection~\ref{Untwisting} we conclude the following.

\begin{corollary}\label{MainResult2}
  If $\Ac = \Oc(G_q)$ is a CQG algebra, then there is a
  characteristic map of the form
  \[ \widetilde{\chi}_q:HC^n(U_q(\Fg), {}^{\s} k) \lra
  HC^{n+1}(\Oc(G_q)) \]
  for every $n\geq 0$.
\end{corollary}

\section{The characteristic map between $U_q(s\ell_2)$ and
  $\Oc(SL_q(2))$}\label{subsect-van-est}

In this section we show the non-triviality of the characteristic map
between cohomologies of $U_q(s\ell_2)$ and $\Oc(SL_q(2))$. We compare
the classes we obtain in its image by the classes computed in
\cite{MasuNakaWata90}.

\subsection{The coordinate algebra $\Oc(SL_q(2))$}
Let us begin with the definition of the coordinate algebra $\Oc(SL_q(2))$ of the quantum group $SL_q(2)$. By \cite[Subsect. 4.1.2]{KlimSchm-book}, it is the algebra generated by
\begin{equation}
\bf{t} = \left(\begin{array}{cc}
                 a & b \\
                 c & d
               \end{array}
\right)
\end{equation}
subject to the relations
\begin{align}\label{aux-Oq-relations}
\begin{split}
& ab=qba,\quad ac=qca,\quad ad=da+(q-q^{-1})bc,\\
& bc=cb,\quad bd=qdb,\quad cd=qdc,\quad ad-qbc=1.
\end{split}
\end{align}
The rest of the Hopf algebra structure is given by
\begin{equation}\label{aux-coalg-Oq}
\D({\bf t}) = {\bf t}\ot {\bf t},\quad \ve({\bf t})=1,\quad S({\bf t})=\left(\begin{array}{cc}
                                                                               d & -q^{-1}b \\
                                                                               -qc & a
                                                                             \end{array}
\right).
\end{equation}

Moreover, it is proved in \cite[Thm. 4.21]{KlimSchm-book} that
\begin{equation}\label{aux-Uq-Oq-pairing}
\langle K, a\rangle = q^{-1}, \quad \langle K, d\rangle=q,\quad \langle E,c\rangle=\langle F,b\rangle=1
\end{equation}
determines a non-degenerate pairing between the Hopf algebras $U_q(s\ell_2)$ and $\Oc(SL_q(2))$.

\subsection{The quantum characteristic map}
It is shown in \cite[Thm. 4.14]{KlimSchm-book} that there exists a
unique invariant linear functional $h:\Oc(SL_q(2))\lra k$ such that
$h(1)=1$. This is the Haar functional of $\Oc(SL_q(2))$ as defined in
Subsection~\ref{CQGAlgebras}. This functional satisfies
\begin{equation}\label{aux-inv-Haar-SL2}
((\Id\ot h)\circ \D)(x) = h(x)1= ((h\ot \Id)\circ \D)(x)
\end{equation}
for all $x\in \Oc(SL_q(2))$.  More explicitly, by
\cite[Thm. 4.14]{KlimSchm-book},
\begin{equation}\label{aux-Haar-functional-SL2}
h(a^rb^kc^\ell) = h(b^kc^\ell d^r) = 
\begin{cases}
  0,\, & r\neq 0,\,{\rm or}\,k\neq \ell \\
  (-1)^k\frac{q-q^{-1}}{q^{k+1}-q^{-(k+1)}},\, & r=0,\,{\rm and}\,k=\ell.
\end{cases}
\end{equation}

Furthermore, by \cite[Prop. 4.15]{KlimSchm-book} the Haar functional $h:\Oc(SL_q(2))\lra k$ is not central (a trace), instead
\begin{equation}
h(xy) = h(\vartheta(y)x)
\end{equation}
for all $x,y\in\Oc(SL_q(2))$ where $\vartheta:\Oc(SL_q(2))\lra \Oc(SL_q(2))$ is the automorphism given by \cite[Prop. 4.5]{KlimSchm-book} as
\begin{equation}
\vartheta(a)=q^2 a,\quad \vartheta(b)= b,\quad \vartheta(c)= c,\quad \vartheta(d)=q^{-2} d.
\end{equation}

\begin{lemma}\label{lemma-Haar-vartheta}
  The automorphism $\vartheta:\Oc(SL_q(2))\lra \Oc(SL_q(2))$ can be given by the action of $K^{-1}\in U_q(s\ell_2)$ in the sense that $\vartheta(x) = K^{-1}\rt x\lt K^{-1}$ for any $x\in \Oc(SL_q(2))$.
\end{lemma}

\begin{proof}
In view of the pairing \eqref{aux-Uq-Oq-pairing} we have
\begin{align*}
K^{-1}\rt a\lt K^{-1} 
= & \langle K,S(a\ps{1})\rangle a\ps{2} \langle K,S(a\ps{3})\rangle \\
= & \langle K,S(a)\rangle a \langle K,S(a)\rangle 
    + \langle K,S(b)\rangle c \langle K,S(a)\rangle \\
  & + \langle K,S(a)\rangle b \langle K,S(c)\rangle 
    + \langle K,S(b)\rangle d \langle K,S(c)\rangle \\
= & \langle K,d\rangle^2 a = q^2a = \vartheta(a).
\end{align*}
Similarly, we have
\begin{align*}
K^{-1}\rt b\lt K^{-1} 
= & \langle K,S(b\ps{1})\rangle b\ps{2} \langle K,S(b\ps{3})\rangle \\
= & \langle K,S(a)\rangle a \langle K,S(b)\rangle 
    + \langle K,S(b)\rangle c \langle K,S(b)\rangle \\
  & + \langle K,S(a)\rangle b \langle K,S(d)\rangle 
    + \langle K,S(b)\rangle d \langle K,S(d)\rangle \\
= & \langle K,d\rangle \langle K,a\rangle b = b = \vartheta(b),
\end{align*}
and
\begin{align*}
K^{-1}\rt c\lt K^{-1} 
= & \langle K,S(c\ps{1})\rangle c\ps{2} \langle K,S(c\ps{3})\rangle \\
= & \langle K,S(c)\rangle a \langle K,S(a)\rangle 
    + \langle K,S(d)\rangle c \langle K,S(a)\rangle \\
  & + \langle K,S(c)\rangle b \langle K,S(c)\rangle 
    + \langle K,S(d)\rangle d \langle K,S(c)\rangle \\
= & \langle K,a\rangle \langle K,d\rangle c = c = \vartheta(c),
\end{align*}
and finally
\begin{align*}
K^{-1}\rt d\lt K^{-1} 
= & \langle K,S(d\ps{1})\rangle d\ps{2} \langle K,S(d\ps{3})\rangle \\
= & \langle K,S(c)\rangle a \langle K,S(b)\rangle 
    + \langle K,S(d)\rangle c \langle K,S(b)\rangle \\
  & + \langle K,S(c)\rangle b \langle K,S(d)\rangle 
    + \langle K,S(d)\rangle d \langle K,S(d)\rangle \\
= & \langle K,a\rangle^2 d = q^{-2}d = \vartheta(d).
\end{align*}
as we wanted to show.
\end{proof}

\begin{lemma}
The right coregular action of $\s^{-1}=K^{-1}\in U_q(s\ell_2)$ on $\Oc(SL_q(2))$ is an automorphism of $\Oc(SL_q(2))$.
\end{lemma}

\begin{proof}
We have
\begin{align*}
(fg)\lt \s^{-1} 
= & \langle(fg)\ps{1},\s^{-1}\rangle (fg)\ps{2} \\
= & \langle f\ps{1}g\ps{1},\s^{-1}\rangle f\ps{2}g\ps{2}\\
= & \langle f\ps{1},\s^{-1}\rangle \langle g\ps{1},\s^{-1}\rangle f\ps{2}g\ps{2} \\
= & (f\lt \s^{-1})(g \lt \s^{-1}).
\end{align*}
\end{proof}

\begin{lemma}
The Haar functional $h:\Oc(SL_q(2))\lra k$ is $\ve$-invariant with respect to the left coregular action of $U_q(s\ell_2)$ on $\Oc(SL_q(2))$.
\end{lemma}

\begin{proof}
Via the invariance of \eqref{aux-inv-Haar-SL2}, for any $y\in U_q(s\ell_2)$ and $f\in \Oc(SL_q(2))$ we have
\begin{equation*}
h(y(f))=h(f\ps{1})\langle y, f\ps{2}\rangle = h(f)\langle y, 1\rangle = \ve(y)h(f).
\end{equation*}
\end{proof}

As a result, using Theorem~\ref{MainResult1} we get the following.
\begin{corollary}
For the Hopf algebra $U_q(s\ell_2)$ with the modular pair $(\ve,\s)$ in involution, the Haar functional $h:\Oc(SL_q(2))\lra k$ determines a characteristic homomorphism
\begin{align}\label{aux-char-homom-SL2}
\begin{split}
& \chi_q:HC^\ast(U_q(s\ell_2), {}^{\s} k) \lra HC_{\s^{-1}}^\ast(\Oc(SL_q(2))),\\
& \chi_q(y^1,\ldots,y^n)(f^0,\ldots,f^n):=h(f^0y^1(f^1)\ldots y^n(f^n)),
\end{split}
\end{align}
where for any $x,y\in U_q(s\ell_2)$ and $f\in \Oc(SL_q(2))$, $y(f)(x):=f(xy)$ is the left coregular action.
\end{corollary}

Combining with Corollary~\ref{MainResult2}, we obtain the following result.
\begin{corollary}
For the Hopf algebra $U_q(s\ell_2)$ with the modular pair $(\ve,\s)$ in involution, the Haar functional $h:\Oc(SL_q(2))\lra k$ determines a characteristic homomorphism
\begin{align}\label{aux-char-homom-SL-2}
\widetilde{\chi}_q:HC^\ast(U_q(s\ell_2), {}^{\s} k) \lra HC^{\ast+1}(\Oc(SL_q(2))).
\end{align}
\end{corollary}

\subsection{The non-triviality of the quantum characteristic map}
In order to discuss the non-triviality of the characteristic
homomorphism \eqref{aux-char-homom-SL2} we recall the results of
\cite{MasuNakaWata90}.  First define
\begin{equation*}
d(t):=\left\{\begin{array}{cc}
               d^t   & \text{ if } t\geq 0 \\
               a^{-t} & \text{ if } t<0,
             \end{array}
\right.
\end{equation*}
and
\begin{equation*}
(x;q)_n:=(1-x)(1-qx)\cdots(1-q^{n-1}x).
\end{equation*}
In \cite{MasuNakaWata90} it is calculated that
\begin{align}\label{aux-MasuNakaWata-results} 
HC^n(\Oc(SL_q(2))) = &
    \begin{cases}
     k[\tau_{\rm even}]  
     \oplus \displaystyle \bigoplus_{i,j,k,\ell>0} k[\tau_a^i] 
     \oplus k[\tau_b^j] 
     \oplus k[\tau_c^k] 
     \oplus k[\tau_d^\ell] & \text{ if } n=0\\
     k S^{\lfloor n/2 \rfloor}[\tau_{\rm even}]  & \text{ if $n>0$ is even}\\
     k S^{\lfloor n/2 \rfloor}[\tau_{\rm odd}]   & \text{ if $n$ is odd}
  \end{cases}
\end{align}
where
\begin{align}
\begin{split}
& \tau_a^l(d(t)b^mc^n) = \d_{t,-l}\d_{m,0}\d_{n,0},\\
& \tau_b^l(d(t)b^mc^n) =\d_{t,0}\d_{m-n,l}\frac{q^l-1}{q^{l+2n}-1}(-q)^n,\\
& \tau_c^l(d(t)b^mc^n) =\d_{t,0}\d_{n-m,l}\frac{q^l-1}{q^{l+2m}-1}(-q)^m,\\
& \tau_d^l(d(t)b^mc^n) =\d_{t,l}\d_{m,0}\d_{n,0},\\
& \tau_{\rm even}(d(t)b^mc^n) = \d_{t,0}\d_{m,0}\d_{n,0},
\end{split}
\end{align}
and finally
\begin{align}
\begin{split}
\tau_{\rm odd}(d(t)b^mc^n, d(\widetilde{t})b^{\widetilde{m}}c^{\widetilde{n}}) 
= & 0 \qquad \text{ if } t+\widetilde{t}\neq 0,\\
\tau_{\rm odd}(a^tb^mc^n, d^{\widetilde{t}}b^{\widetilde{m}}c^{\widetilde{n}}) 
= & \tau_{\rm odd}(d^tb^mc^n, a^{\widetilde{t}}b^{\widetilde{m}}c^{\widetilde{n}}) \\
= & (n-m)(-q)^{n+\widetilde{n}}q^{t(\widetilde{m}+\widetilde{n})}\frac{(q^2;q^2)_t}{(q^{2(n+\widetilde{n})};q^2)_{t+1}}\d_{t,\widetilde{t}}
\d_{m+\widetilde{m},n+\widetilde{n}}.
\end{split}
\end{align}

We are now ready to compute the images, under the characteristic homomorphism \eqref{aux-char-homom-SL-2}, of the Hopf-cyclic classes \eqref{aux-Hopf-cyclic-class-Uq-sl2}.

\begin{proposition}\label{MNW1}
The classes $[\widetilde{\chi}_q(E)],[\widetilde{\chi}_q(KF)]\in HC^2(\Oc(SL_q(2)))$ are nontrivial.
\end{proposition}

\begin{proof}
By the definition \eqref{aux-Haar-functional-SL2} of the Haar functional $h:\Oc(SL_q(2))\lra k$, we have
\begin{align*}
\begin{split}
&\widetilde{\chi}_q(E)(x_0, x_1, x_2) = \\
& -\chi_q(E)(x_0(1-\s)(x_1),\,x_2) - \chi_q(E)(x_0x_1,\,(1-\s)(x_2)) + \chi_q(E)(x_0,\,(1-\s)(x_1)x_2) \\
& =h(x_0(1-\s)(x_1)E(x_2)) - h(x_0x_1E(1-\s)(x_2)) + h(x_0E((1-\s)(x_1)x_2)). \\
\end{split}
\end{align*}
We consider the element $\omega=b\ot c^2\ot a - a \ot b \ot c^2\in \Oc(SL_q(2))^{\ot\,3}$ on which any coboundary vanishes, that is, for any cyclic 1-cocycle $\vp:\Oc(SL_q(2))^{\ot\,2}\lra k$,
\begin{equation*}
b\vp(\omega) = \vp(bc^2\ot a) - \vp(b\ot c^2a) + \vp(ab\ot c^2) - \vp(ab\ot c^2) + \vp(a\ot bc^2) - \vp(c^2a\ot b) = 0.
\end{equation*}
Hence it follows from
\begin{align*}
 &\widetilde{\chi}_q(E)(\omega) = h(b(1-\s)(c^2)E(a)) - h(bc^2E(1-\s)(a)) + h(bE((1-\s)(c^2)a)) \\
 & = (2q^{-2}+q^{-1}-1)h(b^2c^2) = (2q^{-2}+q^{-1}-1)\frac{q-q^{-1}}{q^3-q^{-3}} \neq 0
\end{align*}
that $[\widetilde{\chi}_q(E)]\neq 0$. Using the element $\omega=c\ot b^2\ot d - d \ot c \ot b^2\in \Oc(SL_q(2))^{\ot\,3}$, we similarly arrive at $[\widetilde{\chi}_q(KF)]\neq 0$.
\end{proof}

In view of \eqref{aux-MasuNakaWata-results} we conclude the following.

\begin{corollary}\label{MNW2}
We have $[\widetilde{\chi}_q(KF)]=[\widetilde{\chi}_q(E)] = [S(\tau_{\rm even})]\in HC^2(\Oc(SL_q(2)))$.
\end{corollary}

\section{The $q$-Index Cocycle for the Standard Podleś Sphere}\label{quantumHomogeneousSpaces}

In this section we discuss the equivariant generalization of \eqref{aux-char-homom-SL2}, and we capture the Schmüdgen-Wagner index cocycle of \cite{SchmWagn04} in the image of the equivariant characteristic map.

\subsection{QUE algebra $U_q(su_2)$ and CQG algebra $\Oc(SU_q(2))$}
Let $\Oc(SU_q(2))$ be the coordinate Hopf algebra of the compact
quantum group $SU_q(2)$, see \cite[Subsect. 4.1.4]{KlimSchm-book}. Following the
notation of \cite{SchmWagn04}, let also $U_q(su_2)$ be the Hopf algebra
generated by $E,F,K,K^{-1}$ subject to the relations
\begin{equation*}
KK^{-1}=K^{-1}K=1, \quad KE=qEK,\quad FK=qKF,\quad EF-FE=\frac{K^2-K^{-2}}{q-q^{-1}},
\end{equation*}
whose Hopf algebra structure is given by
\begin{align*}
  & \D(K)=K\ot K,\quad \D(K^{-1})=K^{-1}\ot K^{-1}, \\
  & \D(E)=E\ot K + K^{-1}\ot E,\quad \D(F)=F\ot K + K^{-1}\ot F, \\
  & \ve(K)=\ve(K^{-1})=1,\quad \ve(E)=\ve(F)=0,\\
  & S(K)=K^{-1},\quad S(E)=-qE,\quad S(F)=-q^{-1}F.
\end{align*}
We note that, in the terminology of \cite{KlimSchm-book}, it is the Hopf algebra $\breve{U}_q(s\ell_2)$.

We also note that the non-degenerate pairing between the Hopf algebras
$U_q(su_2)$ and $\Oc(SU_q(2))$ is given by
\begin{equation*}
\langle K^{\pm 1}, d\rangle =\langle K^{\mp 1}, a\rangle = q^{\pm 1/2}, \quad \langle E,c\rangle = \langle F,b\rangle=1.
\end{equation*}
Then, $\Oc(SU_q(2))$ is a left (and a right) $U_q(su_2)$-module algebra via the coregular action.

\subsection{The standard Podleś Sphere}
The coordinate $\ast$-algebra $\Oc(S_q^2)$ of the standard Podleś sphere \cite{Podl87} is the unital $\ast$-algebra with three generators $A=A^\ast,B,B^\ast$ with the relations
\begin{equation*}
BA=q^2AB,\quad AB^\ast=q^2B^\ast A,\quad B^\ast B=A-A^2,\quad BB^\ast=q^2A-q^4A^2.
\end{equation*}
It is also possible to view it as the $K$-invariant subalgebra
\begin{equation*}
\Oc(S_q^2)=\{x\in \Oc(SU_q(2))\mid x\lt K = x\}
\end{equation*}
of $\Oc(SU_q(2))$. We also recall from \cite{SchmWagn04} that for the Haar state $h$ on $\Oc(SU_q(2))$, we have $h(xy)=h((K^{-2}(y)\lt K^{-2})x)$ for any $x,y\in\Oc(SU_q(2))$. Hence, for any $x,y\in\Oc(S_q^2)$ we have $h(xy)=h(\s(y)x)$ with $\s=K^{-2}$.

\subsection{Equivariant Hopf-cyclic cohomology and its actions}
Let us also recall from \cite[Lemma 4.1]{SchmWagn04} that
\begin{equation*}
h(R_F(x)R_E(y)) = q^2h(R_E(x)R_F(y)), 
\end{equation*}
for all $x,y\in \Oc(S_q^2)$.  As a result, the functional
$\tau:\Oc(S_q^2)^{\ot\,3}\lra k$, defined for all $x,y,z\in \Oc(S_q^2)$ as
\begin{align*}
\tau(x,y,z):=h(xR_F(y)R_E(z)-q^2xR_E(y)R_F(z)),
\end{align*}
is a nontrivial $\s$-twisted cyclic 2-cocycle, \ie $[\tau]\in HC^2_\s(\Oc(S_q^2))$.

On the other hand, we recall the cup product construction defined in \cite[Thm. 3.3]{RangSutl-IV}. Let $\Hc$ be a Hopf algebra, ${\Kc}\subseteq {\Hc}$ a cocommutative Hopf subalgebra, and finally $V$ and $N$ are SAYD modules over ${\Kc}$ and ${\Hc}$ respectively. It is proved in \cite[Thm. 3.1]{RangSutl-IV} that
\begin{equation*}
C_{\Kc}({\Hc},V,N) := \bigoplus_{p\geq 0}C^p_{\Kc}({\Hc},V,N),\qquad C^p_{\Kc}({\Hc},V,N):= \Hom_{\Kc}(V,\,N\ot_{\Hc} {\Hc}^{\ot\,p+1})
\end{equation*}
is a cocyclic module, computing the equivariant Hopf-cyclic cohomology $HC_{\Kc}({\Hc},V,N)$, via
\begin{align*}
  d_i(\phi)(v) = &
  \begin{cases}
     \part_i(\phi(v)), & \text{ if } 0\le i\le p,\\
     \part_{p+1}(\phi(v\ns{0}))\lt S(v\ns{-1}), & \text{ if } i=p+1,
  \end{cases}\\
s_j(\phi)(v)   = & 
     \s_j(\phi(v)), \quad \text{ for } \quad 0\le j\le p-1,\\
t_p(\phi)(v)   = & \tau_p(\phi(v\ns{0}))\lt S(v\ns{-1}),
\end{align*}
where the morphisms  $\p_i,\s_j,$ and $\tau$ are those given by \eqref{aux-coface-C}, \eqref{aux-codeg-C} and \eqref{aux-cyclic-C}, and 
\begin{equation*}
(n \ot_\Hc h^0\odots h^p)\lt u = n \ot_\Hc h^0\odots h^pu.
\end{equation*}
We recall also that $\phi \in C^p_{\Kc}({\Hc},V,N)$ if
\begin{equation}\label{aux-equivariancy}
\phi(v\cdot u) = \phi(v)\cdot u,
\end{equation}
for any $u\in \Kc$ where 
\begin{equation*}
(n\ot_\Hc h^0\odots h^p)\cdot u := n\ot_\Hc h^0u\ps{1}\odots h^pu\ps{p+1}.
\end{equation*}
Employing the notation $\phi(v) =: \phi(v)\snb{-1}\ot_{\Hc}  \phi(v)\snb{0} \odots \phi(v)\snb{p}$ for $\phi \in C^p_{\Kc}( {\Hc},V,N)$, 
let us set $\Psi: C^p_{\Kc}( {\Hc},V,N)\ot C^p_{\Hc}({\Ac},N)\longrightarrow  C^p_{\Kc}({\Ac},V)$ as
  \begin{align*}
\Psi(\phi\ot\psi)(v\ot x_0\odots x_p)= \psi\big( \phi(v)\snb{-1}\ot  \phi(v)\snb{0}(x_0)\ot \phi(v)\snb{1}(x_1) \odots \phi(v)\snb{p}(x_p)\big).
  \end{align*}
Then the equivariant characteristic map is given by the cup product
\begin{equation}\label{aux-equiv-cup}
HC^p_{\Kc}( {\Hc}, V,N)\ot HC^q_{\Hc}({\Ac},N)\ra HC^{p+q}_{\Kc}({\Ac},V), \qquad [\phi]\cup[\vp]:=\Psi({\rm Sh}(\phi\ot \vp)).
\end{equation}
using the shuffle map ${\rm Sh}:{\rm Tot}\lra {\rm Diag}$, from the total of the tensor product of the complexes $C^\ast_{\Kc}( {\Hc}, V,N)$ and $C^\ast_{\Hc}({\Ac},N)$ to the diagonal. Adopting the notation of \cite{KhalRang04}, the shuffle map is given by
\begin{equation*}
{\rm Sh}:{\rm Tot}^n\lra {\rm Diag}^n,\qquad {\rm Sh} = \sum_{p+q=n}\nb_{p,q},
\end{equation*}
where
\begin{equation*}
\nb_{p,q} = \sum_{\mu\in {\rm Sh}_{q,p}}(-1)^\mu d_{\wbar{\mu}(p+q)}\ldots d_{\wbar{\mu}(p+1)}\p_{\wbar{\mu}(p)}\ldots \p_{\wbar{\mu}(1)}.
\end{equation*}

\subsection{The $q$-index cocycle for the standard Podleś sphere}
Let us take $\Hc=U_q(su_2)$, $\Kc=k[\s,\s^{-1}]$, $N={}^{\s^{-1}}k$, $V={}^{\s}k$ and $\Ac=\Oc(S_q^2)$. On the next proposition we compute the $q$-index cocycle in the equivariant Hopf-cyclic cohomology.

\begin{proposition}
Let $F\in C^2_\Kc(\Hc,V,N)$ be given by 
\begin{equation}\label{index-cocycle}
F(\one):=\one\ot_\Hc 1\ot (KF\ot EK^3- EK\ot K^3F) - (q^3-q)^{-1}\one\ot_\Hc 1\ot 1 \ot K^4.
\end{equation}
Then, $[F]\in HC^2_\Kc(\Hc,V,N)$, \ie $F$ is an equivariant cyclic 2-cocycle.
\end{proposition}

\begin{proof}
Let us first show that $F$ is indeed $\Kc$-equivariant. For any $K^m$ with $m\in \Zb$ we have
\begin{align*}
F(\one)\cdot K^m = & 
\one\ot_\Hc K^m\ot (KFK^m\ot EK^{3+m} - EK^{m+1}\ot K^3FK^m)\\
  & - (q^3-q)^{-1}\one\ot_\Hc K^m\ot K^m \ot K^{4+m} \\
= & \one\cdot K^m\ot_\Hc 1\ot (KF\ot EK^3- EK\ot K^3F) \\
  & - (q^3-q)^{-1}\one\ot_\Hc 1\ot 1 \ot K^4 = F(\one\cdot K^m).
\end{align*}
Let us next show that $F$ is a Hochschild 2-cocycle. To this end we note that
\begin{align*}
b(\one\ot_\Hc 1\ot EK\ot K^3F) = & 
 \one\ot_\Hc 1 \ot 1\ot EK\ot K^3FK^2 \\
 & - \one\ot_\Hc 1\ot (1\ot  EK + EK\ot K^2)\ot K^3FK^2 \\
 & + \one\ot_\Hc 1\ot EK\ot (K^2\ot K^3FK^2 + K^3F\ot K^2K^2)\\
 & - \one\ot_\Hc 1\ot EK\ot K^3F \ot K^2K^2\\
 = & 0,
\end{align*}
and similarly that $b(\one\ot_\Hc 1\ot KF\ot EK^3) = 0$.  As a result,
$b(F)=0$. We next observe that
\begin{align*}
t(\one\ot_\Hc 1\ot EK\ot K^3F)
 = & \one\ot_\Hc EK\ot K^3F \ot K^2K^2 \\
 = & \one\ot_\Hc E\ot K^3FK^{-1} \ot K^3 \\
 = & -q^{-2}\one\ot_\Hc 1\ot EFK^2 \ot K^4 - \one\ot_\Hc 1\ot KF \ot EK^3,
\end{align*}
where we used \eqref{aux-equivariancy} in the second equality, and that
\begin{align*}
t(\one\ot_\Hc 1\ot KF\ot EK^3)
  = & \one\ot_\Hc KF\ot EK^3\ot K^4 \\
 = & -q^{-2}\one\ot_\Hc 1\ot FEK^2\ot K^4 - \one\ot_\Hc 1\ot EK\ot K^3F.
\end{align*}
As a result,
\begin{align*}
 t(\one\ot_\Hc 1\ot (KF\ot EK^3- EK\ot K^3F))
 = & \one\ot_\Hc 1\ot (KF\ot EK^3- EK\ot K^3F) \\
   & + q^{-2}\one\ot_\Hc 1\ot (EF-FE)K^2 \ot K^4 \\
 = & \one\ot_\Hc 1\ot (KF\ot EK^3- EK\ot K^3F) \\
   & + (q^3-q)^{-1}\one\ot_\Hc 1\ot K^4 \ot K^4 \\
   & - (q^3-q)^{-1}\one\ot_\Hc 1\ot 1 \ot K^4.
\end{align*}
On the other hand, $b(\one\ot_\Hc 1\ot 1 \ot K^4) = 0$, and
$t(\one\ot_\Hc 1\ot 1 \ot K^4) = \one\ot_\Hc 1\ot K^4 \ot K^4$.
Hence, we have
\begin{align*}
b(\one\ot_\Hc 1\ot (KF\ot EK^3- EK\ot K^3F) - (q^3-q)^{-1}\one\ot_\Hc 1\ot 1 \ot K^4) = 0
\end{align*}
and
\begin{align*}
t(\one\ot_\Hc 1\ot & (KF\ot EK^3- EK\ot K^3F) - (q^3-q)^{-1}\one\ot_\Hc 1\ot 1 \ot K^4)\\
= & \one\ot_\Hc 1\ot (KF\ot EK^3- EK\ot K^3F) - (q^3-q)^{-1}\one\ot_\Hc 1\ot 1 \ot K^4.
\end{align*}
We thus conclude that
\begin{equation*}
F = \one\ot_\Hc 1\ot (KF\ot EK^3- EK\ot K^3F) - (q^3-q)^{-1}\one\ot_\Hc 1\ot 1 \ot K^4 \in C^2_\Kc(\Hc,V,N)
\end{equation*}
is an equivariant cyclic 2-cocycle.
\end{proof}

Now, using the equivariant cup product \eqref{aux-equiv-cup} we obtain the following version of the Schmüdgen-Wagner 2-cocycle \cite{SchmWagn04}, see also \cite{Hadf07}.

\begin{corollary}
There is a nontrivial $\s^{-1}$-twisted cyclic 2-cocycle $\tau$ on $\Oc(S_q^2)$ such that
\begin{equation}\label{Schmudgen-Wagner-index}
\tau(x_0,x_1,x_2) = h(x_0KF(x_1)EK^3(x_2)) - h(x_0EK(x_1)K^3F(x_2)) - (q^3-q)^{-1}h(x_0x_1K^4(x_2)).
\end{equation}
\end{corollary}

\begin{proof}
We obtain the cocycle \eqref{Schmudgen-Wagner-index} by the cup product \eqref{aux-equiv-cup} of the Hopf-cyclic 0-cocycle $[h] \in HC^0_{U_q(su_2)}(\Oc(S_q^2),{}^{\s^{-1}}k)$ with the equivariant 2-cocycle \eqref{index-cocycle}.

We next show that it is nontrivial. Following \cite{SchmWagn04}, we consider the element
\begin{align*}
\eta'
 = & q^4B^\ast\ot A \ot B + q^2B\ot B^\ast\ot A + q^2A\ot B\ot B^\ast \\
   & - q^{2}B^\ast\ot B\ot A - q^{2}A\ot B^\ast\ot B 
     - B\ot A\ot B^\ast + (q^6-q^{2})A\ot A\ot A
\end{align*}
on which any Hochschild coboundary of a $\s^{-1}$-twisted cyclic 1-cocycle vanishes. Indeed, for any $\s^{-1}$-twisted cyclic 1-cocycle $\tau'\in HC^1_{\s^{-1}}(\Oc(S_q^2))$,
\begin{equation*}
b\tau'(\eta') = (q^4-q^2)\tau'(A\ot A) = 0.
\end{equation*}
On the other hand, a quick computation yields $\tau(\eta')\neq 0$.
\end{proof}

\bibliographystyle{plain}
\bibliography{references}{}

\end{document}